\documentclass[12pt,a4paper,twoside]{article}
\usepackage{amsmath,amsthm,amssymb,amsfonts,amsbsy}

\usepackage{color, colortbl} 
\usepackage{float} 
\usepackage[margin=2.5cm]{geometry} 
\usepackage{cite}
\usepackage{authblk} 

\theoremstyle{definition}

\theoremstyle{plain}

\newtheorem{lem}[subsection]{Lemma}
\newtheorem{prop}[subsection]{Proposition}

\newcommand{\beq}{\begin{eqnarray}}
\newcommand{\eeq}{\end{eqnarray}}

\newcommand{\beqs}{\begin{eqnarray*}}
\newcommand{\eeqs}{\end{eqnarray*}}

\allowdisplaybreaks

\usepackage{mathtools}

\title{\bf A Note on the Modified Albertson Index}

\author{
   \large Shumaila Yousaf$^{a}$, Akhlaq Ahmad Bhatti$^{a}$, Akbar Ali$^{b}$
}

\affil{ \normalsize
    { \it Department of Sciences and Humanities, \\ \it National University of Computer and Emerging Sciences,\\ \it B-Block, Faisal Town, Lahore, Pakistan}\\
    E-mail: {\tt shumaila.yousaf@uog.edu.pk, akhlaq.ahmad@nu.edu.pk}\\
    \vspace{2mm}
    { \it Knowledge Unit of Science,\\  \it University of Management and Technology, Sialkot-Pakistan}
    \\E-mail: {\tt akbarali.maths@gmail.com}

}

\begin{document}
\maketitle

\begin{abstract}
The \textit{modified Albertson index}, denoted by $A\!^*\!$, of a graph $G$ is defined as $A\!^*\!(G)=\sum_{uv\in E(G)} |(d_{u})^{2}- (d_{v})^{2}|$, where $d_u$, $d_v$ denote the degrees of the vertices $u$, $v$, respectively, of $G$ and $E(G)$ is the edge set of $G$. In this note, a sharp lower bound of $A\!^*$ in terms of the maximum degree for the case of trees is derived. The $n$-vertex trees having maximal and minimal $A\!^*$ values are also characterized here. Moreover, it is shown that $A\!^*\!(G)$ is non-negative even integer for every graph $G$ and that there exist infinitely many connected graphs whose $A\!^*$ value is $2t$ for every integer $t\in\{0,3,4,5\}\cup\{8,9,10,\cdots\}$.

\end{abstract}

\maketitle



\section{Introduction}
All the graphs considered in this note are simple and finite. Sets of vertices and edges of a graph $G$ will be denoted by $V(G)$ and $E(G)$, respectively. Degree of a vertex $u$ and the edge connecting the vertices $u,v\in V(G)$ will be denoted by $d_u$ and $uv$, respectively. Let $N(u)$ be the set of all those vertices of $G$ which are adjacent to $u$. By an $n$-vertex graph, we mean a graph with $n$ vertices. The graph theoretical terminology not defined here, can be found from some standard books of graph theory, like \cite{Harary-69,Bondy-08}.

For a graph $G$, the imbalance of the edge $uv\in E(G)$, denoted by $imb(uv)$, is defined as $|d_u-d_v|$. The idea of the imbalance of an edge was actually appeared implicitly in \cite{Albertson-91} within the study of Ramsey graphs. Using the concept of imbalance, Albertson \cite{2} defined the following graph invariant
$$
A(G)=\sum_{uv\in E(G)} imb(uv)\,
$$
and named it as the \textit{irregularity} of $G$; however, several researchers \cite{Gutman-05,Hansen-05,Reti-MATCH-18,Matejic-18,Ali-16,Furtula-13} referred it as the \textit{Albertson index} and we do the same in this paper. Detail about the mathematical properties of the Albertson index can be found in the recent papers \cite{Chen-18,Ashrafi-19,Nasiri-19,Nasiri-IJMC-18} and related references listed therein.

This note is devoted to establish some properties of the following modified version of the Albertson index
$$
A\!^*\!(G)=\sum_{uv\in E(G)} |(d_{u})^{2}- (d_{v})^{2}|\,.
$$
We propose to call the graph invariant $A\!^*$ as the \textit{modified Albertson index}.

\section{Main Results}

Firstly, we prove two results concerning the modified Albertson index of trees; one of these results is related to a sharp lower bound of $A\!^*$ in terms of maximum degree and the second one is an extremal result, in which we characterize the $n$-vertex trees having maximal and minimal $A\!^*$ values.

\begin{prop}\label{thm-new-2}
If $T$ is a tree with maximum degree $\Delta$ then $A\!^*\!(T)\ge \Delta(\Delta^2 - 1)$ with equality if and only if $T$ is isomorphic to either a path or a tree containing only one vertex of degree greater than 2.

\end{prop}

\begin{proof}
The result is obvious for $\Delta\le2$ and hence we assume that $\Delta\ge3$. If $v\in V(T)$ has maximum degree then there are $d_v$ pendant vertices namely $w_1,w_2,\cdots,w_{d_v}\,$ in $T$ such that the paths $v-w_1,v-w_2,\cdots,v-w_{d_v}\,$ are pairwise internally disjoint. If the path $v-w_1$ has length greater than 1, suppose that $w_{1,1},w_{1,2},\cdots, w_{1,r}$ are the internal vertices of the path $v-w_1$. Then,
\begin{align*}
&\big|(d_v)^2-(d_{w_{1,1}})^2\big| + \big|(d_{w_{1,1}})^2-(d_{w_{1,2}})^2\big| +\cdots + \big|(d_{w_{1,r}})^2 - (d_{w_{1}})^2\big|\\
&\ge \big[(d_v)^2-(d_{w_{1,1}})^2\big] + \big[(d_{w_{1,1}})^2-(d_{w_{1,2}})^2\big] +\cdots + \big[(d_{w_{1,r}})^2 - (d_{w_{1}})^2\big]=\Delta^2 -1\,.
\end{align*}
We note that the equality
\[
\big|(d_v)^2-(d_{w_{1,1}})^2\big| + \big|(d_{w_{1,1}})^2-(d_{w_{1,2}})^2\big| +\cdots + \big|(d_{w_{1,r}})^2 - (d_{w_{1}})^2\big| = \Delta^2 -1
\]
holds if and only if the degrees of successive vertices along the path from $v$ to $w_1$ decrease monotonously (not
necessarily strictly). Similarly, for $i=2,\cdots,r$, the sum of contributions of edges to $A\!^*\!(T)$ along the path $v-w_i$ is at least $\Delta^2 -1$ with equality if and only if the degrees of successive vertices along the path from $v$ to $w_i$ decrease monotonously (not necessarily strictly), and hence the desired result follows.

\end{proof}

\begin{prop}\label{thm-new-3}
For $n\ge5$, if $T$ is an $n$-vertex tree different from the path $P_n$ and star $S_n$, then $A\!^*\! (P_n) < A\!^*\! (T) <A\!^*\! (S_n)$.

\end{prop}

\begin{proof}
The inequality $A\!^*\! (P_n) < A\!^*\! (T)$ follows from Proposition \ref{thm-new-2}. To prove the inequality $A\!^*\! (T) <A\!^*\! (S_n)$, we note that for any two vertices $u,v\in V(T)$, it holds that $|(d_u)^2 - (d_v)^2| \le |(n-1)^2 - 1|$ with equality if and only if one of the vertices $u,v$ has degree 1 and the other has degree $n-1$. But, $T$ does not contain any vertex of degree $n-1$ and hence
\[
A\!^*\! (T)=\sum_{uv\in E(T)}|(d_u)^2 - (d_v)^2| < (n-1)|(n-1)^2 - 1|= A\!^*\! (S_n).
\]

\end{proof}

Let $u$ be a fixed vertex of $G$. We partition the set $N(u)$ as follows: $L(u)=\{v\in N(u):~d_v<d_u\}$, $E(u)=\{v\in N(u):~d_v=d_u\}$ and $G(u)=\{v\in N(u):~d_v>d_u\}$. The number of elements in $L(u)$, $E(u)$ and $G(u)$ ar denoted by $l_u$, $e_u$ and $g_u\,$, respectively. Clearly, $d_u=l_u+e_u+g_u\,$.
Now, we will prove that the modified Albertson index $A\!^*$ is non-negative even integer for every graph; but, before proving this fact, we derive the following useful result first.

\begin{lem}\label{thm-new-1}
If $u$ and $v$ are non-adjacent vertices in a graph $G$ such that $d_u\ge d_v$ then
\[
A\!^*\!(G+uv) = A\!^*\!(G) + 3d_u(d_u+1) + d_v(d_v-1) -2[(2d_u+1)g_u + (2d_v+1)g_v]\,.
\]

\end{lem}

\begin{proof}
We consider the difference
\begin{align*}
A\!^*\!(G+uv) - A\!^*\!(G)
&=(d_u +1)^2 - (d_v +1)^2 \\
& \quad + \sum_{x\in N(u)} \bigg( \big|(d_u +1)^2 - (d_x)^2\big| - \big|(d_u)^2 - (d_x)^2\big|\bigg)\\
& \quad + \sum_{y\in N(v)} \bigg( \big|(d_v +1)^2 - (d_y)^2\big| - \big|(d_v)^2 - (d_y)^2\big|\bigg)\,.
\end{align*}
Now, using the facts $N(u)=L(u)\cup E(u)\cup G(u)$, $N(u)=L(v)\cup E(v)\cup G(v)$ and then after simplifying, we arrive at
\begin{align*}
A\!^*\!(G+uv) - A\!^*\!(G)
&=(d_u - d_v)(d_u +d_v +2)+ (2d_u+1)(e_u +l_u - g_u)\\
& \quad + (2d_v+1)(e_v +l_v - g_v)\,,
\end{align*}
which is equivalent to
\begin{align*}
A\!^*\!(G+uv) - A\!^*\!(G)
&=3d_u(d_u+1) + d_v(d_v-1) -2[(2d_u+1)g_u + (2d_v+1)g_v]\,.
\end{align*}

\end{proof}

\begin{prop}\label{cor-new-1}
The modified Albertson index $A\!^*$ of every graph is a non-negative even integer.

\end{prop}

\begin{proof}
Let $G$ be any graph. By definition, $A\!^*(G)\ge0$ with equality if and only if every component of $G$ is regular. The result obviously holds if $G$ is the complete graph and hence we assume that $G$ is not isomorphic to a complete graph.
We prove the result by induction on the number of edges of $G$. If $G$ is the edgeless graph then $A\!^*\!(G)=0$ and hence the induction starts. Let $u$ and $v$ be non-adjacent vertices of $G$ such that $d_u\ge d_v$. Then, by Lemma \ref{thm-new-1}, it holds that
\begin{equation}\label{Eq-cor-new-1}
A\!^*\!(G+uv) = A\!^*\!(G) + 3d_u(d_u+1) + d_v(d_v-1) -2[(2d_u+1)g_u + (2d_v+1)g_v]\,.
\end{equation}
By induction hypothesis, $A\!^*\!(G)$ is even and hence from Equation \eqref{Eq-cor-new-1}, it follows that $A\!^*\!(G+uv)$ is even. This completes the induction and hence the proof.

\end{proof}

\noindent
{\bf Transformation 1.} Let $uv$ be an edge of a graph $G$ satisfying $d_u=d_v=3$. Let $G'$ be the graph obtained from $G$ by inserting a new
vertex $x\not\in V(G)$ of degree 2 on the edge $uv$.\\

Finally, we prove that there exist infinitely many connected graphs whose modified Albertson index is $2t$ for every integer $t\in\{0,3,4,5\}\cup\{8,9,10,\cdots\}$. For this, we need the following two lemmas whose proofs are straightforward.

\begin{lem} \label{lem-new2}
If $G$ and $G'$ are the two graphs specified in Transformation 1, then $A\!^*\!(G') = A\!^*\!(G)+10$.
\end{lem}

\begin{lem} \label{lem-new1}
Let $uv$ be an edge of a graph $G$ satisfying one of the following conditions
\begin{enumerate}
  \item $d_u=1$ and $d_v\ge 2$;
  \item at least one of the vertices $u,v$ has degree 2.
\end{enumerate}
If $G'$ is the graph obtained from $G$ by inserting a new
vertex $x\not\in V(G)$ of degree 2 on the edge $uv$, then $A\!^*\!(G') = A\!^*\!(G)$.
\end{lem}

\begin{prop}\label{thm-new-4}
For every integer $t\in\{0,3,4,5\}\cup\{8,9,10,\cdots\}$, there exist infinitely many connected graphs whose $A\!^*$ value is $2t$.

\end{prop}

\begin{figure}[ht]
		\centering
			\includegraphics[width=1.0\textwidth]{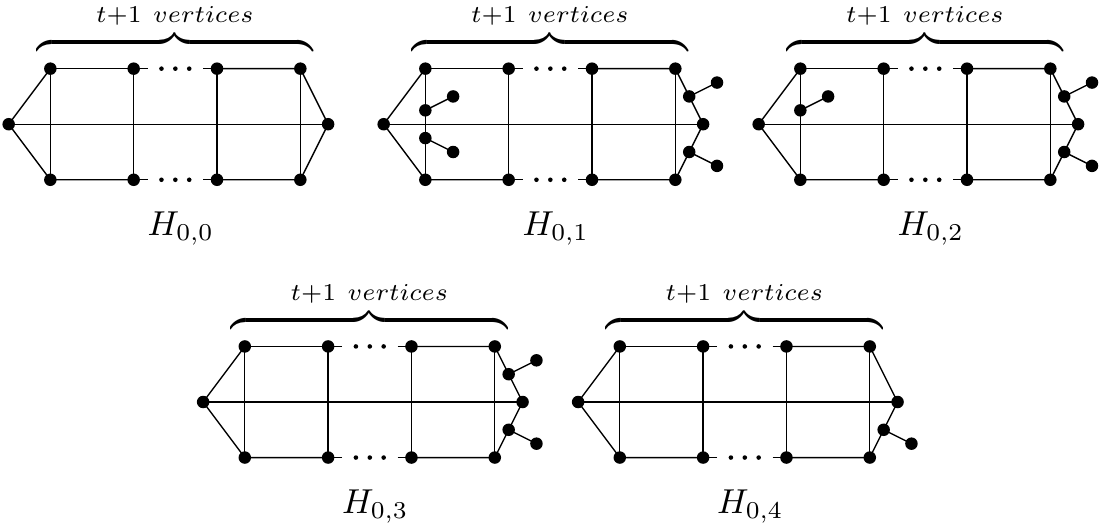}
		\caption{The graphs $H_{0,0}$, $H_{0,1}$, $H_{0,2}$, $H_{0,3}$ and $H_{0,4}$, used in the proof of Proposition \ref{thm-new-4}.}
  \label{fig1}	
  \end{figure}

\begin{proof}

Let $H_{0,0}$ be the cubic graph shown in Figure \ref{fig1}. Obviously, $H_{0,0}$ has $3(t+2)$ edges and its $A\!^*$ value is 0.
Also, we consider the graphs $H_{0,1}$, $H_{0,2}$, $H_{0,3}$ and $H_{0,4}$ (which are obtained from $H_{0,0}$) depicted in Figure \ref{fig1}; their $A\!^*$ values are 32, 24, 16 and 8, respectively. For $j=0,1,2,3,4$ and $1\le i < 3(t+2)$, let $H_{i,j}$ be the graph obtained from $H_{i-1,j}$ by applying Transformation 1. Then,
\[
 A\!^*\!(H_{i,j})= \begin{cases}
    10i & \hbox{if $j=0$;} \\
    2(5i-4j+20) & \hbox{otherwise.}
  \end{cases}
\]
We yet need to find the graphs with $A\!^*$ values 22 and 6. The $A\!^*$ value of the 3-vertex path graph $P_3$ is 6. Let $H$ be the graph obtained from the 5-vertex complete graph $K_5$ by inserting a new vertex $x\not\in V(K_5)$ of degree 2 on an edge of $K_5$. If $H'$ is the graph obtained from $H$ by attaching a new vertex $y\not\in V(H)$ to the vertex $x\in V(H)$, then $A\!^*\!(H')=22$. Until now, we have found a single graph having modified Albertson index $2t$ for each $t\in\{0,3,4,5\}\cup\{8,9,10,\cdots\}$. Now, by using the transformation specified in Lemma \ref{lem-new1}, we get infinitely many graphs with the same $A\!^*$ value, corresponding to each of the graphs $H_{i,j}$, $P_3$, $H'$.




\end{proof}

\end{document}